\documentclass[10pt]{amsart}
\usepackage[utf8x]{inputenc}

\usepackage{scalerel,amssymb}

\usepackage[font={footnotesize,it}]{caption}
\usepackage{ulem}
\usepackage{dsfont}
\usepackage{amsmath}
\usepackage{amssymb}
\usepackage{graphicx}
\usepackage{amssymb}
\usepackage{amsfonts}
\usepackage{soul}
\usepackage{mathpazo}
\usepackage{color}
\usepackage{yfonts}
\usepackage{paralist}
\usepackage{stmaryrd}
\usepackage{amsxtra}
\usepackage{latexsym,mathrsfs}
\usepackage{soul}
\usepackage{mathabx}
\usepackage{float}
\usepackage{epsfig, enumerate}
\usepackage{color}
\usepackage{hyperref}

\usepackage{rotating}
\newtheorem{theorem}{Theorem}[section]
\newtheorem{lemma}[theorem]{Lemma}
\newtheorem{Remark}[theorem]{Remark}

\newtheorem{definition}[theorem]{Definition}
\newtheorem{prop}[theorem]{Proposition}

\numberwithin{theorem}{section}
\numberwithin{equation}{section}

\usepackage{amsthm,regexpatch}

\makeatletter

\def\endcproof{%
  \renewcommand\qedsymbol{\openbox\rlap{\textsubscript{\std@currentclaim}}}%
  \endproof
}
\makeatother

\newtheorem{claim}{Claim}[theorem]

\makeatletter
\xpretocmd{\claim}{\std@patch@thm}{}{}
\newcommand{\std@patch@thm}{%
  \makeatletter
  \regexpatchcmd{\@thm}
    {(\c{refstepcounter}\cB.\cP.2\cE.)}
    {\1\c{std@savelabel}}
    {}{}%
  \makeatother
}
\newcommand{\std@savelabel}{\xdef\std@currentclaim{\@currentlabel}}
\makeatother

\title[Generalized Theta Graph]{On the Turán Number of Generalized Theta Graphs}

\author{Xiao-Chuan Liu}
\address[Liu]{Instituto de Matemática da Universidade Federal de Alagoas,
	Av. Lourival Melo Mota, S/N, Maceió, Brasil} 
\email{lxc1984@gmail.com}

\author{Xu Yang}
\address[Yang]{Instituto de Computação da Universidade Federal de Alagoas,
	Av. Lourival Melo Mota, S/N, Maceió, Brasil} 
\email{yang@ic.ufal.br}

\begin{document}
\maketitle{}
\begin{abstract} Let $\Theta_{k_1,\cdots,k_\ell}$ denote the generalized theta graph, which consists of $\ell$ internally disjoint paths with lengths $k_1,\cdots, k_{\ell}$, connecting two fixed vertices. 
We estimate the corresponding extremal number $\text{ex}(n,\Theta_{k_1,\cdots,k_\ell})$. 
When the lengths of all paths have the same parity and at most one path has length 1, $\text{ex}(n,\Theta_{k_1,\cdots,k_\ell})$ is $O(n^{1+1/k^\ast})$, where $2k^\ast$ is the length of the smallest cycle in $\Theta_{k_1,\cdots,k_\ell}$. 
We also establish matching lower bound in the particular case of $\text{ex}(n,\Theta_{3,5,5})$.\end{abstract}

\section{Introduction}
For a graph $H$, define the extremal number $\text{ex}(n,H)$ as the maximum number of edges a graph on $n$ vertices can have without containing a copy of $H$. This number is also referred to as Tur{\'a}n number because of the pioneering work of Tur{\'a}n which initiated the whole area (see \cite{Turan1941}). One of the central problems in this area is to determine the order of the extremal number for a graph. 
The celebrated Erd\H{o}s-Stone-Simonovits theorem states that if the chromatic number of $H$ is denoted  by $\chi(H)$, then 
\begin{equation}
\text{ex}(n,H)=\big ( 1-\frac{1}{\chi(H)-1} +o(1)\big ) \frac {n^2}{2}.
\end{equation}

One calls an extremal problem degenerate, if the corresponding extremal number has order $o(n^2)$. 
Therefore, this theory focuses on forbidding bipartite graphs. Degenerate extremal graph theory recently has seen lots of exciting developments. 
See the recent survey~\cite{Furedi2013history} for a treatment of both the history 
as well as the  state of the art of this theory. 
A very interesting class of bipartite graphs is that of even cycles. Bondy and Simonovits showed in ~\cite{Bondy1974cycles} that $\text{ex}(n,C_{2k})= O(n^{1+1/k})$. 
Although these bounds were conjectured to be of the correct order, matching lower bounds were only found for the cases $k=2,3,5$ (see~\cite{wenger1991extremal} by Wenger for constructions of all these three cases. See also ~\cite{Conlon_cycles_C4C6C10} by Conlon for a geometric interpretation of these examples). However, the simplest unclear case of $\text{ex}(n,C_8)$ 
still seems to be very difficult. 

In order to better understand even cycles, people also look at a related class of graphs called theta graphs. With time, the study of theta graphs also 
became interesting in its own right, and recently it has drawn a lot of attention. 
By definition, the graph $\theta_{k,\ell}$ is obtained by fixing two vertices and connecting them with $\ell$ internally disjoint paths of length $k$. Note that in this notation, $\theta_{k,2}$ is simply a synonym for $C_{2k}$. 
Already in the 80s, 
Faudree and Simonovits in~\cite{Faudree1983} showed that for any $k,\ell$, the extremal number $\text{ex}(n,\theta_{k, \ell})=O(n^{1+1/k})$. 
On the other hand, some lower bounds were obtained only very recently. 
Based on the method of random polynomials invented by Bukh in ~\cite{Bukh2014random},
Conlon showed in ~\cite{Conlon2019graphs} that for any $k$, for all sufficiently large $\ell$,
$\text{ex}(n,\theta_{k,\ell})= \Omega(n^{1+\frac 1k})$. Note here the largeness of $\ell$ is not explicit.  

In this work, we focus on a larger class of graphs, often referred to as  generalized theta graphs. 
More precisely, we make the following definition. 
\begin{definition}\label{general_theta_graph} Let $k_1,\cdots,k_\ell$ be positive integers, with the same parity, in which $1$ appears at most once. Define the \textbf{generalized theta graph}, 
denoted by $\Theta_{k_1,\cdots,k_\ell}$, 
to be the graph obtained by fixing two vertices $w$ and $w'$, which are connected by $\ell$ internally disjoint paths with lengths $k_1,\cdots,k_\ell$, respectively.
\end{definition}
\begin{Remark} The parity requirement makes these graphs bipartite. 
\end{Remark}

The main result of this paper is the following upper bound. 
\begin{theorem}\label{main} 
Fix positive integers $k_1,\cdots,k_\ell$ with the same parity, in which $1$ appears at most once. Then, 
\begin{equation}
	\text{ex}(n,\Theta_{k_1,\cdots,k_\ell})=O(n^{1+\frac {1}{k^\ast}}),
\end{equation}
where $k^\ast= \frac 12 \text{min}_{1\leq i<j\leq \ell} (k_i+k_j)$.
\end{theorem}

We remark that our estimate aims to find the correct exponent, and does not focus too much on the constant hidden in the big $O$ notation.
In some recent developments, 
people want to  more carefully understand the dependence of the constant on the graph. For example, Bukh and Jiang showed in \cite{bukh2017bound} 
that $\text{ex}(n,C_{2k})$ is upper bounded by $80 \sqrt k \log k \ n^{1+1/k}$, which was further improved by He 
in~\cite{new_C2k_bound} to 
 $\big (16 \sqrt {5k\log k} +o(1) \big) \ n^{1+1/k}$.
In the same spirit, Bukh and Tait~\cite{Bukh_theta} 
showed that for theta graphs, 
$\text{ex}(n,\theta_{k,\ell}) \leq c_k \ell^{1-1/k}n^{1+1/k}$ for some constant $c_k$ depending on $k$. 
In the upcoming project, we also intend to combine techniques from~\cite{Bukh_theta} 
with the ideas from the present paper to give a more precise estimate on dependence of the coefficient on the path lengths in a generalized theta graph.  

After Theorem~\ref{main}, one can raise natural questions for matching lower bounds of the new family of graphs we have considered. 
Notice that, towards the very difficult problem of finding matching lower bound of 
$\text{ex}(n,C_8)$, Verstra{\"e}te and Williford established in a recent paper~\cite{verstraete2019graphs} that 
$\text{ex}(n,\Theta_{4,4,4})= \Omega(n^{5/4})$. 
Observing a recent construction in~\cite{Conlon_cycles_C4C6C10} by Conlon, which in turn was a rephrasing of an algebraic construction 
by Wenger in~\cite{wenger1991extremal},
here we give a quick proof of matching lower bound for a very similar graph $\Theta_{3,5,5}$, and  establish the following.
\begin{theorem}\label{Thm335} $\text{ex}(n,\Theta_{3,5,5})=\Theta(n^{5/4})$. \end{theorem}

In the rest of the paper, after preparing ourselves with basic notation and several lemmas in Section~\ref{notation}, the proofs of Theorem~\ref{main} and Theorem~\ref{Thm335} will be given in Section~\ref{theorem12proof} and Section~\ref{proof_335}, respectively. 

\section{Basic Notation and Useful Lemmas}\label{notation}
Write $G=\big (V(G),E(G)\big )$ for a graph $G$ with 
its vertex set $V(G)$ and edge set $E(G)$. 
 Throughout this paper, the graphs we consider are all simple, undirected and connected. 
If there is no substantial difference, 
we ignore rounding when we need a number to be an integer. 
A special notation that we will use is as follows. For a positive real number $a\geq 1$, 
we let $K_{a,1}$ denote the star graph consisting of a vertex called center, and $\lfloor a \rfloor$ other vertices joined to it.

We begin with the following classical lemma. The simple proof is provided for completeness.

\begin{lemma}\label{embedding_tree} 
Suppose $G$ is a graph on $n$ vertices with 
$|E(G)|\geq \ell n$. Then for $n\geq 2\ell+1$, $G$ contains a copy of any 
tree $T_{\ell+1}$ with $\ell+1$ vertices. 
Moreover, when $G$ is bipartite, the embedding can be done such that, 
one can prescribe a vertex in the tree and 
embed it in any preferred part in the bipartition. 
\end{lemma}
\begin{proof} First 
\textbf{claim} that $G$ admits a subgraph $H$ whose minimal degree is at least 
$\ell$. To show the claim, we induct on $n$, with base case $n_0=2\ell+1$ such that the complete graph $K_{n_0}$ satisfies the conclusion. 
Now the induction hypothesis is that any graph on $n-1$ 
vertices with at least $\ell(n-1)$ edges admits a subgraph $H$ with minimal degree at least $\ell$. Then we look at any graph $G$ on $n$ vertices with at least $n\ell$ edges. 
If there is any vertex $v \in V(G)$ whose degree is strictly smaller than $\ell$, then we form $G'$ by deleting this vertex. Then $G'$ on $n-1$ vertices has at least $\ell(n-1)$ edges, which must contain a subgraph $H$ with minimal degree at least $\ell$ by induction hypothesis. If there is no such vertex, we are also done since $G$ already has minimal degree at least $\ell$. 
Finally, for any tree $T_{\ell+1}$ with $\ell+1$ vertices, we can greedily embed $T_{\ell+1}$ into $H$. 

It is left to check the second statement.
After we obtained the subgraph H with minimal degree at least $\ell$, 
in the final embedding process, 
we can start by embedding 
the prescribed vertex in the preferred part, 
and the rest of the process follows unchanged. 
\end{proof}
\begin{lemma}\label{extracting_stars}
Let $G$ be a bipartite graph on the vertex bipartition $V\cup W$. Suppose $|V|=m$, $|W|\geq mp$ for some $p\geq 1$.  For all $v\in V$, $\text{deg}(v) \leq Cp$ for some constant $C\geq 1$. For all $w\in W$, $\text{deg}(w)>0$. Then, there exists a subgraph which is a disjoint union of  at least $\frac{1}{C+1}m$ copies of  $K_{p/C,1}$'s, whose centers are all in $V$. 
\end{lemma}
\begin{proof} 
The proof is by a  simple greedy algorithm.
We start by choosing a vertex $v_0\in V$ and finding a copy of $K_{p/C,1}$ centered at $v_0$. 
Then we delete both $v_0$ and all the vertices adjacent to $v_0$, 
and obtain a new bipartite graph called $G_1$ on the bipartition $V_1\cup W_1$. 
Note the number of vertices in $W_1$ is more than $(m-C)p$, each of which 
is adjacent to some vertex of $V_1$. 
There are 
$(m-1)$ vertices in $V_1$.
As long as $m\geq C+1$, and thus $(m-1)\frac{p}{C} \leq (m-C) p$, 
it follows that there exists at least one vertex $v_1\in V_1$, whose degree in $W_1$ is at least  $\frac{p}{C}$.

Inductively, suppose in $G$ we have found a union of $k$ copies of $K_{p/C,1}$ centered at vertices $v_0,v_1,\cdots,v_{k-1}\in V$, then we delete all these vertices in $V$, obtaining vertex set $V_k$, and then delete from $W$ all the vertices which are adjacent to at least one of $v_j, j=0,\cdots,k-1$, obtaining the new vertex set $W_k$. This gives a new bipartite graph $G_k$. Now, the size of $W_k$ is at least $(m-Ck)p$, 
and there are exactly $m-k$ vertices in $V_k$. 
As long as $k\leq \frac{1}{C+1}m$, and thus $(m-k)\frac{p}{C}\leq (m-Ck) p$, 
we can find one vertex $v_k\in V_k$, whose degree in $W_k$ is at least $\frac{p}{C}$. 
The process stops only when $k>\frac{1}{C+1}m$, when we have already embedded the subgraph we wanted. 
\end{proof}

We will need the following well known reduction lemma, 
 which provides a subgraph with a sufficiently large vertex set for which the degree of every vertex is under control. 
 \begin{lemma}\label{degree_control}[Lemma 5 of~\cite{Bukh_theta}. See also Proposition 2.7 of~\cite{Jiang2012turan} and Theorem 1 of~\cite{erdos1965some}] 
For any $\alpha>0$, there exist $\varepsilon_0 >0$ and $C_0>1$, such that, any 
graph $G$ with $|V(G)|=n$ and 
$|E(G)|= K n^{1+\alpha}$ contains a subgraph $H$ such that
 $|V(H)|=m =\Omega(n^{\varepsilon_0})$ and every vertex $v\in V(H)$ has degree 
 $\text{deg}_{H}(v)\in [\frac {1}{C_0} K m^{\alpha}, C_0 K m^{\alpha}]$.
\end{lemma}

\section{Proof of the Theorem~\ref{main}}\label{theorem12proof}
Hereafter, the integer $n$ is always considered to be sufficiently large. The proof of Theorem~\ref{main} reduces to the following proposition. 
\begin{prop}\label{proposition_general} 
For any $C_0>1$ there exists $M>0$ with the following property. 
Let $G$ be a $\Theta_{k_1,\cdots,k_\ell}$-free bipartite graph, for which the degree of every vertex belongs to the interval $[\frac{1}{C_0}n^{1/k^\ast},C_0n^{1/k^\ast}]$. 
Then $|V(G)| \geq \frac {1}{M} n +1$.
\end{prop}
\begin{proof}[Proof of Theorem~\ref{main} using Proposition~\ref{proposition_general}]
For $\alpha=\frac {1}{k^\ast}$, 
Lemma~\ref{degree_control} produces constants $C_0$ and $\varepsilon_0$, and Proposition~\ref{proposition_general} produces the constant $M$.
Suppose for contradiction that 
for some sufficiently large $n$,  a graph $G$ on $n$ vertices has more than 
$2M^{1/k^\ast}n^{1+1/k^\ast}$ edges. 
Then $G$ contains a bipartite subgraph $G'$ 
with more than 
$M^{1/k^\ast}n^{1+1/k^\ast}$ edges. 
By Lemma~\ref{degree_control}, a subgraph $H$ of $G'$ on 
$m \gtrsim  n^{\varepsilon_0}$ vertices satisfies that  each vertex in $H$ 
has degree lying in the interval $[\frac{1}{C_0} (M m)^{1/k^\ast}, C_0 (M m)^{1/k^\ast}]$. Applying Proposition~\ref{proposition_general}, one obtains that $|V(H)| \geq m+1$, which is a contradiction. 
\end{proof}

In the rest of this section, we prove Proposition~\ref{proposition_general}.
To make the exposition clearer, we split this proof into three subsections. 
\subsection{Preliminary Considerations and Setup of the Proof}
Recall Definition~\ref{general_theta_graph} and fix 
the graph 
$\Theta_{k_1,\cdots,k_\ell}$, which consists of two fixed vertices,  namely $w$ and $w'$,
connected by $\ell$ internally disjoint paths 
of lengths $k_1,\cdots,k_\ell$ (edge numbers), respectively. 
Up to reordering the indices, we can simply assume $k_1\leq k_2\leq \cdots \leq k_\ell$, 
and thus $k^\ast$ in Theorem \ref{main} can be written as $k^\ast=(k_1+k_2)/2$. Note if $k_1=k_2$, then the conclusion follows from the main result of \cite{Faudree1983} (i.e. Theorem 2). So we assume $k_1<k^\ast< k_2$.

Suppose $G$ is a connected 
bipartite graph and $r\in V(G)$ is arbitrarily chosen and fixed as \textbf{the root}. 
We will write $L_G^r(i)$ for the set of vertices which have distance $i$ with the root $r$. When there is no confusion about the host graph $G$ and/or root $r$, we can simplify as $L_G^r(i)=L_G(i)=L(i)$. In particular, $L(0)=\{r\}$. 
For any vertex $v\in L(i), u\in L(i+1)$, 
if $u$ and $v$ are adjacent, we call $u$ a {\bf child} of $v$ and $v$ a {\bf parent} of $u$.
For $u\in L(j)$ and $v\in L(i)$ with $j>i$, 
$u$ is called a \textbf{descendant} of $v$ 
if their distance is $j-i$. In this case, 
$v$ is an \textbf{ancestor} of $u$. 
We further make the following definition. 
\begin{definition}\label{regular_almost_tree} 
Given integers $1\leq s\leq k$ and  real number $d>0$, 
a bipartite graph $G$ with a root $r \in V(G)$ and layers $L(j), j=1,\cdots,k$,
is  said to \textbf{restrict to a regular almost-tree of type 
$(d,s)$ (with respect to the root r)}, if the following hold.
\begin{enumerate}
\item every $v\in \bigcup_{j=0}^{s-1}L(j)$ has exactly  $\lfloor d\rfloor$ children, and each vertex $v' \in \bigcup_{j=1}^{s-1}L(j)$ has exactly one parent. 
\item for any $v_1\in L(1)$, $G[\{r\} \cup \bigcup_{j=1}^{s-1} L(j)]$ is 
isomorphic with 
 $G^{v_1,s}$, where  $G^{v_1,s}$ is the induced subgraph  of $G$ 
by $v_1$ and all its descendants until the layer $L(s)$. 
\end{enumerate}
If further every vertex $v\in L(s)$ also  has only one parent, then we say the graph $G$ \textbf{restricts to a regular tree of type} $(d,s)$.
\end{definition}

The following lemma is useful to ``grow a regular tree" into higher layers.
We postpone its proof 
to the appendix due to its elementary nature. 
\begin{lemma}\label{regular_tree} 
For any $C_0,\ C_1>1$, there exists a constant  $K$ depending on $C_0$ and $C_1$ such that the following holds. 
Let $1\leq s<k$, and let $n$ be sufficiently large and $d=\frac{1}{C_0}n^{1/k}$. 
Suppose a bipartite graph $G$ has a root $r\in V(G)$ and the corresponding layers $L(j),j=1\, \cdots,k$, 
satisfying the following conditions. 
\begin{enumerate} 
\item[(A)]  $G$ restricts to a regular tree of type 
$(d,s)$.
\item[(B)] for any $v\in L(s)$, the number of children of $v$ in $L(s+1)$ 
belongs to the interval $[d,C_0^2d]$.
\item[(C)] the induced bipartite subgraph 
 $H=G[L(s)\cup L(s+1)]$ satisfies
\begin{equation}
|E(H)| \leq C_1 |V(H)|.
\end{equation}
\end{enumerate}
Then $G$ has a subgraph 
$G^\ast$  which restricts to a regular tree of type $(\frac {1}{K} d,s+1)$.
\end{lemma}
\begin{proof}See Appendix.
\end{proof}

Hereafter, let $G$ be a bipartite and $\Theta_{k_1,\cdots,k_\ell}$-free graph and for every vertex $v\in G$, $\text{deg}(v)\in [\frac{1}{C_0}n^{1/k^\ast},C_0n^{1/k^\ast}]$. 

\begin{definition}\label{badset_definition} 
For all $i=1,\cdots,k^\ast-1$, we define
$\mathfrak B^{(i)}$ as the set of vertices in $L_G(i)$ which have at least 
$\frac {1}{2C_0}n^{1/k^\ast}$ parents in $L_G(i-1)$. 
\end{definition}

We define each set $\mathfrak B^{(i)}$ in the original graph $G$. These sets can be taken as the first type of ``bad sets". For each $i=1,\cdots,k^\ast-1$, by the degree condition, there are at most 
$C_0^{i}n^{i/k^\ast}$ edges in the induced subgraph $G[L(i)\cup L(i-1)]$. 
Therefore one has the trivial bound
\begin{equation}\label{trivial_bound_of_B}
|\mathfrak B^{(i)}| \leq 2C_0^{i+1} n^{\frac{i-1}{k^\ast}}.
\end{equation}

The general idea of Subsections 3.2 and 3.3 is as follows. We will define several kinds of ``bad sets", and prove that their sizes are small compared to the corresponding layer, so that we can delete them to obtain bigger and bigger regular almost-trees until $k^\ast-1$ layers and derive a contradiction.
In particular, in Subsection 3.2, we will define the second kind of ``bad sets", which are vertices with many children fallen in $\mathfrak B^{(i)}$. This is the first part of the induction step, where we deal with the layers $L(i)$ for $i\leq k_1+2$.
Later in Subsection 3.3, we will define the third kind of ``bad sets", which consist of so-called thick vertices. We will do the second part of induction with the layers $L(i)$ for $k_1+2\leq i \leq k^\ast-1$.  

\subsection{First Part of Induction Step}
The following lemma is useful when we need to repeatly check condition $(C)$ in Lemma \ref{regular_tree}.

\begin{lemma}\label{first_embed}
Let $H$ be a graph which restricts to a regular almost-tree of type $(\frac{1}{M}n^{1/k^\ast},s+1)$, $1\leq s<k_1+1$ and $M$ is a constant. Moreover, $H$ is $\Theta_{k_1,\cdots, k_{\ell}}$- free. Let $U$ and $W$ be any subsets in $L_H(s)$ and $L_H(s+1)$, respectively. Let $R$ be the bipartite graph $R=H[U\cup W]$ and $C_1=|V(\Theta_{k_1,\cdots, k_\ell})|$. Then $|E(R)|<C_1|V(R)|$. 
\end{lemma}

\begin{proof}
Suppose otherwise, that is, the average degree of $R$ is at least $2C_1$. Then there is a  subgraph $R_1$ with minimal degree at least $C_1$. We will embed a copy of $\Theta_{k_1,\cdots,k_\ell}$ in $H$ to reach a 
contradiction.  For this, recall $\Theta_{k_1,\cdots,k_\ell}$ is seen as two vertices $w$ and $w'$ connected by $\ell$ internally disjoint paths. Let $T$ denote the subgraph of $\Theta_{k_1,\cdots,k_\ell}$ induced by all the vertices at distance at least $s$ with $w$, which is a tree. In particular, $w'$ belongs to $T$. 
Then we see $\Theta_{k_1,\cdots,k_\ell} \backslash T$ is an $(s-1)$-subdivided $\ell$-star centered at $w$. We next embed $T$ into the graph $R_1$ 
with the following properties. 
\begin{enumerate}[(a)]
	\item embed all the leaves of $T$ in $U$.
	\item the embedded image of $T$ intersecting $L_H(s)$ belongs to descendants of pairwise distinct 
	vertices in $L_H(1)$.
\end{enumerate}
Property (a) is ensured by Lemma~\ref{embedding_tree}. 
For property (b), 
we observe that property (2) in Definition \ref{regular_almost_tree} says that
for every vertex in $L_H(s+1)$, its neighbours in $L_H(s)$
are descendants of pairwise distinct vertices in $L_H(1)$. Therefore, since in $R_1$ every vertex has degree at least $C_1$, by applying the greedy algorithm 
in Lemma~\ref{embedding_tree}, 
every time when we embed a vertex of $T$ in $L_H(s)\cap R_1$,
we have at least $C_1$
choices whose ancestors in $L_H(1)$ are distinct. Then property (b) follows immediately. 

Now we have embedded the tree $T$ into $H$,
satisfying properties (a) and (b). 
Consider two situations. 
If $s<k_1$, then we look at all the embedded leaves of $T$, and 
trace back to $r$ through its ancestors. 
If $s=k_1$, we need to consider all the embedded leaves of $T$ together with the embedded image of $w'$, 
and then trace back to $r$. In both cases, 
we can embed $(s-1)$-subdivided star  $\Theta_{k_1,\cdots,k_\ell} \backslash T$ with $w$ embedded in $r$ and therefore embed the graph $\Theta_{k_1,\cdots,k_\ell}$.
\end{proof}

\begin{lemma}\label{first_type_induction} 
There exists a constant $M_{k_1+1}$, such that $G$ contains a subgraph $H$ which 
restricts to a regular almost-tree of type 
$(\frac{1}{M_{k_1+1}} n^{1/k^\ast},k_1+2)$.
\end{lemma}
\begin{proof} For $0\leq s\leq k_1+1$, we will construct subgraphs of $G$ which  restrict to regular almost-trees of type $(\frac{1}{M_s} n^{1/k^\ast},s+1)$. We prove this by induction. 
In the base case $s=0$, $G$ restricts to a trivial  regular almost-tree of type $(\frac{1}{C_0}n^{1/k^\ast},1)$.
Now suppose for any  $0\leq s<k_1+1$, we have constructed  $G^{(s)}$  which  restricts to a regular almost-tree of type $(\frac{1}{M_s} n^{1/k^\ast},s+1)$. 
We denote by $L(0)$, $L(1)$, $\cdots$, $L(k^\ast)$ the first $k^\ast+1$ layers of $G^{(s)}$.  

The case $s=0$ is degenerate, and we omit its separate treatment because it is simpler. In the case $s\geq 1$, 
recall definition of $\mathfrak B^{(s+1)}$, and put 
\begin{equation}\label{Bs+1definition}
B_{s+1}=\mathfrak B^{(s+1)} \cap L(s+1).
\end{equation}
For $i=s,\cdots,1$ (in that order), define
\begin{equation}\label{else_Bi}
B_i=\{v\in L(i)  \big |  \text{deg}_{G^{(s)}[L(i) \cup B_{i+1}]}(v)\geq  \frac{1}{2M_s}n^{1/k^\ast}\}. 
\end{equation}
We stress that the definitions of $B_1, \cdots, B_{s+1}$ are within the induction process and the subscripts represent their corresponding layers. Consider the bipartite graph $R=G^{(s)}[B_{s+1}\cup B_{s}]$. 
Take $C_1= |V(\Theta_{k_1,\cdots,k_\ell})|$. By Lemma \ref{first_embed}, we have 

\begin{equation}\label{density_claim}
 	|E(R)| < C_1|V(R)|
\end{equation}

Therefore, we have $|B_s| \times \frac{1}{2M_s}n^{1/k^\ast}\leq C_1 ( | B_s|+ |B_{s+1}|)$, it follows that  
\begin{equation}\label{Bs_estimate}
|B_s| \leq \frac{4M_s C_1}{n^{1/k^{\ast}}} |B_{s+1}|.
\end{equation}
Then for each $i=1,\cdots,s-1$,  since each vertex in $L(i+1)$ has exactly one parent, 
$|B_{i}| \times \frac{1}{2M_s}n^{1/k^\ast}\leq  |B_{i+1}|,$
which implies
\begin{equation}\label{other_biestimate}
|B_i| \leq \frac{2M_s}{n^{1/k^{\ast}}}| B_{i+1}|
\leq  2C_1 \Big(\frac{2M_s}{n^{1/k^\ast}}\Big)^{s+1-i} |B_{s+1}| 
= O(n^{\frac{i-1}{k^\ast}}) \ll |L(i)|,  
\end{equation}
where the equality follows from  (\ref{trivial_bound_of_B}). 
In particular, 
\begin{equation}\label{layer_two_constant}
|B_1|= O(1) \ll |L(1)|.
\end{equation}

We put $G^{(s+1)}=G^{(s)}$ and rename the first $k^\ast+1$ layers as $\{L_{G^{(s+1)}}(i)\}_{i=0}^{k^\ast}$. Firstly, we delete $B_i$ from $L_{G^{(s+1)}}(i)$ for all $i=1,\cdots,s+1$. 
	Remember in $G^{(s)}$, every vertex $v\in L(i)$, for $i=0,\cdots, s$,
 has exactly $\frac{1}{M_s}n^{1/k^\ast}$ children. Now in $G^{(s+1)}$, after the deletion of the sets $B_i$, for $i=0,\cdots,s$,
 every remaining vertex $v\in L_{G^{(s+1)}}(i)$ has at least $\frac{1}{2M_s}n^{1/k^\ast}$ children left. This is true for the case $i=0$ by
 (\ref{layer_two_constant}) and
 the rest cases  $i=1,\cdots,s$ by (\ref{else_Bi}).
	Moreover, each $v\in L_{G^{(s+1)}}(s+1)$ has at least $\frac{1}{2C_0}n^{1/k^\ast}$ children which is of course at least $\frac{1}{2M_s}n^{1/k^\ast}$ by Definition~\ref{badset_definition}.
	
	Therefore, we can delete some more vertices from $L_{G^{(s+1)}}(i)$, $i=s,s-1,\cdots,1$, to update $G^{(s+1)}$ so that	$G^{(s+1)}$ restricts to a regular tree of type $(\frac{1}{2M_s} n^{1/k^\ast}, s)$. 
Note that, we do not delete vertices after the $s$-th layer, so now every vertex in $L_{G^{(s+1)}}(s)$ still has at least $\frac{1}{2M_s}n^{1/k^\ast}$ children and every vertex in $L_{G^{(s+1)}}(s+1)$ has at least $\frac{1}{2C_0}n^{1/k^\ast}$ children. Next we modify $G^{(s+1)}$ in three steps. Note that in all three steps, we only delete vertices in $L_{G^{(s+1)}}(i)$, $i=1,\ldots, s+1$. For the vertices in $L_{G^{(s+1)}}(s+1)$, the number of children does not change. In order to apply Lemma \ref{regular_tree}, the vertices in $L_{G^{(s+1)}}(s+1)$ have many children and condition (B) is satisfied. 

\begin{enumerate}
\item {\bf Grow a regular tree of type $(\frac{1}{M_{s+1}'} n^{1/k^\ast}, s+1)$ from a regular tree of type $(\frac{1}{2M_s} n^{1/k^\ast}, s)$ for some larger constant $M_{s+1}'$.}

 Since every vertex in $L_{G^{(s+1)}}(s)$  has at least $\frac{1}{2M_s}n^{1/k^\ast}$ children, we can delete some edges such that every vertex in $L_{G^{(s+1)}}(s)$ has exactly $\frac{1}{2M_s}n^{1/k^\ast}$ children and $G^{(s+1)}$ restricts to a regular almost-tree of type $(\frac{1}{2M_s} n^{1/k^\ast}, s+1)$. Let $d$ in  Lemma~\ref{regular_tree} be equal to $\frac{1}{2M_s} n^{1/k^\ast}$. The degree of every vertex in $L_{G^{(s+1)}}(s)$ is still upper bounded by $C_0n^{1/k^\ast}<2M_sn^{1/k^\ast}$. Therefore condition (B) of Lemma~\ref{regular_tree} is satisfied. By Lemma \ref{first_embed}, condition (C) of Lemma ~\ref{regular_tree} is satisfied. 
We apply Lemma~\ref{regular_tree} (taking $C_0$ there to be $2M_s$) to update $G^{(s+1)}$ which restricts to
a regular tree of type 
$(\frac{1}{M_{s+1}'} n^{1/k^\ast}, s+1)$, for some constant $M_{s+1}'>2M_s$.

\item {\bf For each $v\in L_{G^{(s+1)}}(1)$, grow a regular tree of type $(\frac{1}{M_{s+1}}n^{1/k^\ast},s+1)$ from a regular tree of type $(\frac{1}{M_{s+1}'}n^{1/k^\ast},s)$,  regarding $v$ as the root.}

 The general idea in this step is that we inductively and alternatively construct regular trees and regular almost-trees from the bottom up by using Lemmas \ref{regular_tree} and \ref{first_embed}.

Let $v'$ be a descendant of $v$ in $L_{G^{(s+1),v}}^v(s-1)$. Let $G^{(s+1),v'}$ denote the subgraph of $G^{(s+1)}$ induced by the vertex $v'$ and all its descendants. Since every vertex in $L_{G^{(s+1),v'}}^{v'}(1)$ has at least $\frac{1}{2C_0}n^{1/k^\ast}$ children, we delete some edges such that every vertex in $L_{G^{(s+1),v'}}^{v'}(1)$ has exactly $\frac{1}{M_{s+1}'}n^{1/k^\ast}$ children and therefore $G^{(s+1),v'}$ restricts to a regular almost-tree of type $(\frac{1}{M_{s+1}'}n^{1/k^\ast},2)$. Clearly,  $G^{(s+1),v'}[L_{G^{(s+1),v'}}^{v'}(1)\cup L_{G^{(s+1),v'}}^{v'}(2)]$ satisfies condition $(C)$ in Lemma \ref{regular_tree} by Lemma \ref{first_embed}. 
By Lemma \ref{regular_tree}, we delete some vertices in $L_{G^{(s+1),v'}}^{v'}(1)$ and some edges such that $G^{(s+1),v'}$ restricts to a regular tree of type $(\frac{1}{M_{s+1}^{(1)}}n^{1/k^\ast},2)$, where $M_{s+1}^{(1)}$ is a constant larger than $M_{s+1}'$. See Figure \ref{figure1}, the right part.

\begin{figure}[H]
	\centering
	\includegraphics[width=1\textwidth]{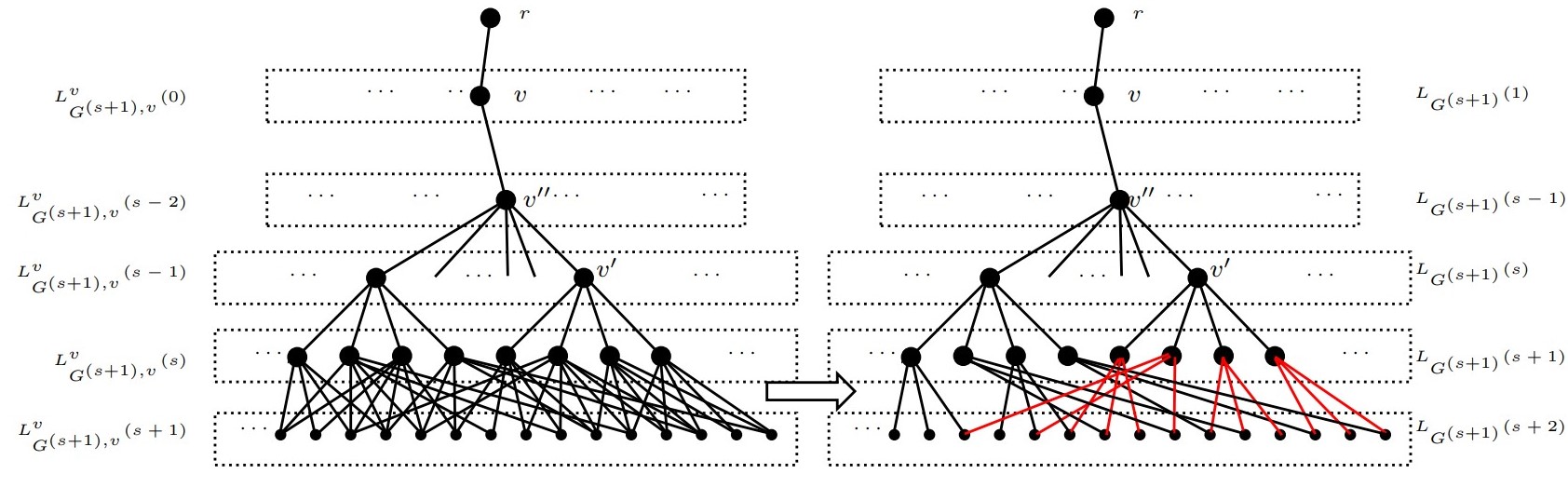}
	\caption{In this figure, we show the base case of Step (2). The left part is before we do Step (2), whereas the right part is after the base case. The names of layers on the left side are with respect to vertex $v$ and right side are with respect to root $r$. It can be seen that on the right side, $G^{(s+1),v'}$ restricts to a regular tree of type $(\frac{1}{M_{s+1}^{(1)}}n^{1/k^\ast},2)$, and as for $v''$, $G^{(s+1),v''}$ restricts to a regular almost-tree of type  $(\frac{1}{M_{s+1}^{(1)}}n^{1/k^\ast},3)$ (up to deleting some descendants).
		The red edges are what are left after the clean up. 
They show that we can clean the last two layers to have the structure of a regular tree.    }\label{figure1}
\end{figure}

Inductively, suppose for some $t$, $1\leq t\leq s-1$, we find a constant $M_{s+1}^{(t)}$, such that for every vertex $v'\in L_{G^{(s+1),v}}^v(s-t)$, $G^{(s+1),v'}$ restricts to a regular tree of type $(\frac{1}{M_{s+1}^{(t)}}n^{1/k^\ast},t+1)$. Now we consider any vertex $v''\in L_{G^{(s+1),v}}^v(s-t-1)$ (when $t=s-1$, $v''=v$.). Note that it has $\frac{1}{M'_{s+1}}n^{1/k^\ast}$ many children. We only keep $\frac{1}{M_{s+1}^{(t)}}n^{1/k^\ast}$ of them so that $G^{(s+1),v''}$ restricts to a regular almost-tree of type $(\frac{1}{M_{s+1}^{(t)}}n^{1/k^\ast},t+2)$. By Lemma \ref{first_embed}, $G^{(s+1),v''}[L_{G^{(s+1),v''}}^{v''}(t+1)\cup L_{G^{(s+1),v''}}^{v''}(t+2)]$ satisfies condition $(C)$ in Lemma \ref{regular_tree}. Clearly conditions $(A)$ and $(B)$ are both satisfied. By Lemma \ref{regular_tree}, $G^{(s+1),v''}$ has a subgraph which restricts to a regular tree of type $(\frac{1}{M_{s+1}^{(t+1)}}n^{1/k^\ast},t+2)$. Especially, when $t=s-1$, let $M_{s+1}=M_{s+1}^{(s)}$, and the regular tree structure is obtained as required.   

\item {\bf Combine all the regular trees of type $(\frac{1}{M_{s+1}}n^{1/k^\ast},s+1)$ into a regular almost-tree of type $(\frac{1}{M_{s+1}}n^{1/k^\ast},s+2)$.} 

After step (2), $G^{(s+1),v}$ restricts to a regular tree of type $(\frac{1}{M_{s+1}}n^{1/k^\ast},s+1)$, for any $v\in L_{G^{(s+1)}}(1)$. But in $L_{G^{(s+1)}}(1)$, we have $\frac{1}{M_{s+1}'}n^{1/k^\ast}$ vertices. We only keep  $\frac{1}{M_{s+1}} n^{1/k^\ast}$ of them. In this way, $G^{(s+1)}$ restricts to a regular almost-tree of type $(\frac{1}{M_{s+1}}n^{1/k^\ast},s+2)$. 
\end{enumerate}  

The above procedure finishes the induction step. This means we obtain $G^{(s+1)}$ which restricts to a regular almost-tree of type $(\frac{1}{M_{s+1}}n^{1/k^\ast},s+2)$.
Finally, the induction stops after the step when we take $s=k_1+1$. Then we can take $H=G^{(k_1+1)}$ to conclude. 
\end{proof}

\subsection{Second Part of Induction Step.}
Assume
a bipartite graph $H$ is $\Theta_{k_1,\cdots,k_\ell}$-free, and 
restricts to a regular almost-tree of type 
$(d, s+1)$, where $k_1+1\leq s\leq k^\ast-1$. 
Here we assume $d$ is an integer.
Fix $r\in V(H)$ the root, and then write 
   \begin{align}
   &  L(1)      =\{v_1^{(1)},\cdots,v_1^{(d)}\},\\
   & L(k_1)      =\{v_{k_1}^{(1)}, \cdots, v_{k_1}^{(d^{k_1})} \},\\ 
   & L(k_1+1)  = \{v_{k_1+1}^{(1)}, \cdots, v_{k_1+1}^{(d^{k_1+1})} \}.
\end{align}
In the layer $L(s)$, we define 
 $\mathfrak  D=\{ D_j\}_{j=1}^{d}$, 
where each $D_j$ is the subset of descendants of the vertex $v_1^{(j)}\in L(1)$. 
Similarly, we denote by $A_t\subset L(s)$, for $t=1,\cdots,d^{k_1}$, the subset of descendants in layer $L(s)$ of the vertex $v_{k_1}^{(t)}\in L(k_1)$. 
We denote by $B_p\subset L(s)$, $p=1,\cdots, d^{k_1+1}$, the set of descendants in layer $L_s$ of the vertex $v_{k_1+1}^{(p)} \in L(k_1+1)$. By assumption $k_1+1\leq s$. 
Each $D_j$ is a disjoint union of $A_t$'s, and each $A_t$ is a disjoint union of $B_p$'s. 
The total number of $A_t$'s is $d^{k_1}$ and each $A_t$ has size $d^{s-k_1}$.
The total number of $B_p$'s is $d^{k_1+1}$, and each $B_p$ has size $d^{s-k_1-1}$.
\begin{definition}\label{gamma(s)definition}
For $i=1,\cdots,\ell-1$, 
put $\tau_i=k_1+k_{i+1}-2s -1$.
Let $\Gamma(s)$ be a graph consisting of $(\ell-1)$ vertex disjoint paths 
$\{P_i\}_{i=1}^{\ell-1}$, where each $P_i$ has edge length $\tau_i$.  Let $\Lambda(s)$ be a tree which is the union of $k_2+\cdots+k_\ell+3\ell$ copies of $P_i$ for $i=1,\cdots,k_{\ell}$ which share one of their endpoints (called the center of $\Lambda(s)$).
\end{definition}
		
\begin{definition}
We call a vertex $w\in L(s+1)$ \textbf{strong}, 
if one can embed $\Lambda(s)$ to $H[L(s)\cup L(s+1)]$ so that the center of $\Lambda(s)$ is sent to $w$ and all the leaves of $\Lambda(s)$ are sent to vertices belonging to pairwise distinct elements in $\mathfrak D$. 
\end{definition}

\begin{definition}\label{thick}
A vertex $u\in L(s)$ is called \textbf{thick}, 
if it has at least one strong neighbour in $L(s+1)$. Let $L^{\text{thick}}(s)\subset L(s)$ denote the set of thick vertices. 
The remaining vertices in 
$L(s)$ are called \textbf{thin}, and are denoted by 
$L^\text{thin}(s): = L(s)\backslash L^{\text{thick}}(s)$. 
\end{definition}

\begin{lemma}\label{thick_thin_discussion} Assume 
a bipartite graph $H$ is $\Theta_{k_1,\cdots,k_\ell}$-free, and $H$ restricts to a regular almost-tree of type 
$(d, s+1)$, where $d$ is an integer. Then 
\begin{equation}
|L^{\text{thick}}(s)| \leq (\ell-2) d^{s-1}.
\end{equation}
\end{lemma}

\begin{proof}
Suppose for contradiction that 
$|L^{\text{thick}}(s)| 
\geq (\ell-2) d^{s-1}+1.$ By pigeonhole principle, since the number of $A_t$'s is $d^{k_1}$, there is a certain $A_t$ containing at least 
$(\ell-2)d^{s-1-k_1}+1 $ thick vertices. 
Now, since each $B_p$ has size $d^{s-1-k_1}$, 
so by pigeonhole principle again, 
we can find thick vertices 
$u_1\in B_{p_1},u_2\in B_{p_2}, \cdots, u_{\ell-1}\in B_{p_{\ell-1}}$, 
such that $\bigcup_{i=1}^{\ell-1} B_{p_i}\subset A_t$, and 
$p_1$, $p_2$, $\cdots$, $p_{\ell-1}$ are distinct. Also assume $A_t\subset D_j$, for some $1\leq j\leq d$. By the definition of regular almost-tree, for $u_1, u_2,\ldots,u_{\ell-1}$, which belong to the same $D_j$, we can find $\ell-1$ distinct strong vertices $w_1,w_2,\ldots,w_{\ell-1}$,
which are adjacent to them, respectively. Moreover, the vertices $u_1, u_2,\ldots,u_{\ell-1}$ belong to a certain $A_t$, which means they are descendants of one single  vertex $v_{k_1}^{(t)}\in L(k_1)$.

Now we can embed the graph $\Theta_{k_1,\cdots,k_\ell}$ as follows (see Figure~\ref{figure}).

{\bf (1) Embed $\Gamma(s)$.}

We start from $u_1$, which has a strong neighbour $w_1 \in L(s+1)$. 
Then we can embed the path $P_1$ with length $k_1+k_2-2s-1$ between $w_1$ and a vertex in $D_{j'}$ with $j'\neq j$. Moreover, we can make sure that the embedded image of $P_1$ does not intersect any  $u_i$ for $i=1,\cdots, \ell-1$, or any $w_i$ for $i=2,3,\cdots,\ell-1$.
 
Inductively, suppose we have already used $u_1,\ldots,u_i$ and hence $w_1,\ldots,w_i$ to successfully embed paths $P_1,P_2,\cdots,P_i$, for $i<\ell-1$. 
 In other words, we make sure the following: 
\begin{enumerate}
	\item the embedded images of $P_1,P_2,\cdots, P_i$ are pair-wise vertex disjoint.
	\item the embedded images end at
	distinct elements of $\mathfrak D$, neither in $D_j$. 
	\item for each $P_{i'}$ in this list, 
	the embedded image of it does not intersect any $u_t$ for $t=1,\cdots,\ell-1$, or any $w_t$ with $t \in \{ 1,\cdots,\ell-1\} \backslash \{i'\}$.
\end{enumerate} 

	In the definition of $\Lambda(s)$, the number  $k_2+\cdots +k_\ell+3\ell$ is taken to be a safe constant, which will be explained later. Note that every $P_t$ has edge length $k_1+k_{t+1}-2s-1<k_{t+1}$, $t=1,\cdots, \ell-1$. Now, starting from the strong vertex $w_{i+1}$ which joins the thick vertex $u_{i+1}$, we aim to embed $P_{i+1}$. With $w_{i+1}$ being strong, it connects with at least $k_2+\cdots+k_{\ell}+3\ell$ internally disjoint paths with lengths $k_1+k_{i+2}-2s-1$, ending at distinct elements of $\mathfrak{D}$. Among these paths, at most $k_2+\cdots+k_{i+1}$ of them intersect at least one of the paths $P_1,\cdots,P_i$. In order to avoid $u_t$, $t=1,\cdots, \ell-1$ and $w_t$, $t\in \{1,\cdots,\ell-1\}\setminus \{i+1\}$, we disregard at most $2\ell$ of the paths. So there are at least $\ell$ paths of lengths $\tau_{i+1}$ which are still available. At most $i$ of them end at the same element of $\mathfrak D$ with one of the embedded paths $P_1, P_2,\ldots, P_i$. Therefore, we can choose one such good path to embed $P_{i+1}$, which finishes the induction step. Eventually, at the end of the induction  we have embedded the forest $\Gamma(s)$ as we wanted. In Figure~\ref{figure}, the blue paths represent the paths $P_1,P_2,\ldots,P_{\ell-1}$.

{\bf (2) Extend $\Gamma(s)$ to $\ell -1$ longer internally disjoint paths.}

Note that $P_1, \cdots, P_{\ell-1}$ end at vertices belonging to pairwise distinct elements of $\mathfrak D$, also different from $D_j$.
 Noticing the structure of regular almost-tree, there are $\ell-1$ internally disjoint paths $P'_1,\cdots,P'_{\ell-1}$ starting from the end vertices of $P_1,\cdots,P_{\ell-1}$ and ending at $r$. In Figure~\ref{figure}, we illustrate the paths $P'_1,\cdots,P'_{\ell-1}$ with red paths. 
 
{\bf (3) Find $\ell$ internally disjoint paths between $r$ and $v_{k_1}^{(t)}$.} 

 Noticing again  the structure of regular almost-tree, we can find $\ell-1$ disjoint paths $Q_1,\cdots,Q_{\ell-1}$ starting from $v_{k_1}^{(t)}$ and ending at $u_1,\cdots,u_{\ell-1}$, respectively. In Figure~\ref{figure}, we illustrate the paths $Q_1,\cdots,Q_{\ell-1}$ with green paths. 

Therefore, between vertices $r$ and $v_{k_1}^{(t)}$, we find $\ell$ internally disjoint paths with length $k_1,\cdots, k_\ell$. The first path is the path through the regular almost-tree with length $k_1$. 
The other $\ell-1$ paths are $Q_i\cup u_iw_i\cup P_i\cup P'_i$, for $i=1,2,\cdots,\ell-1$. We illustrate this procedure in Figure~\ref{figure}. By simply adding the edge length of each part, we can see that each path has edge length $k_{i+1}$. Therefore, it gives an embedding of the graph $\Theta_{k_1,\cdots,k_\ell}$ into $H$, which is a contradiction. 
\end{proof}

\begin{figure}
	\centering
	\includegraphics[width=0.7\textwidth]{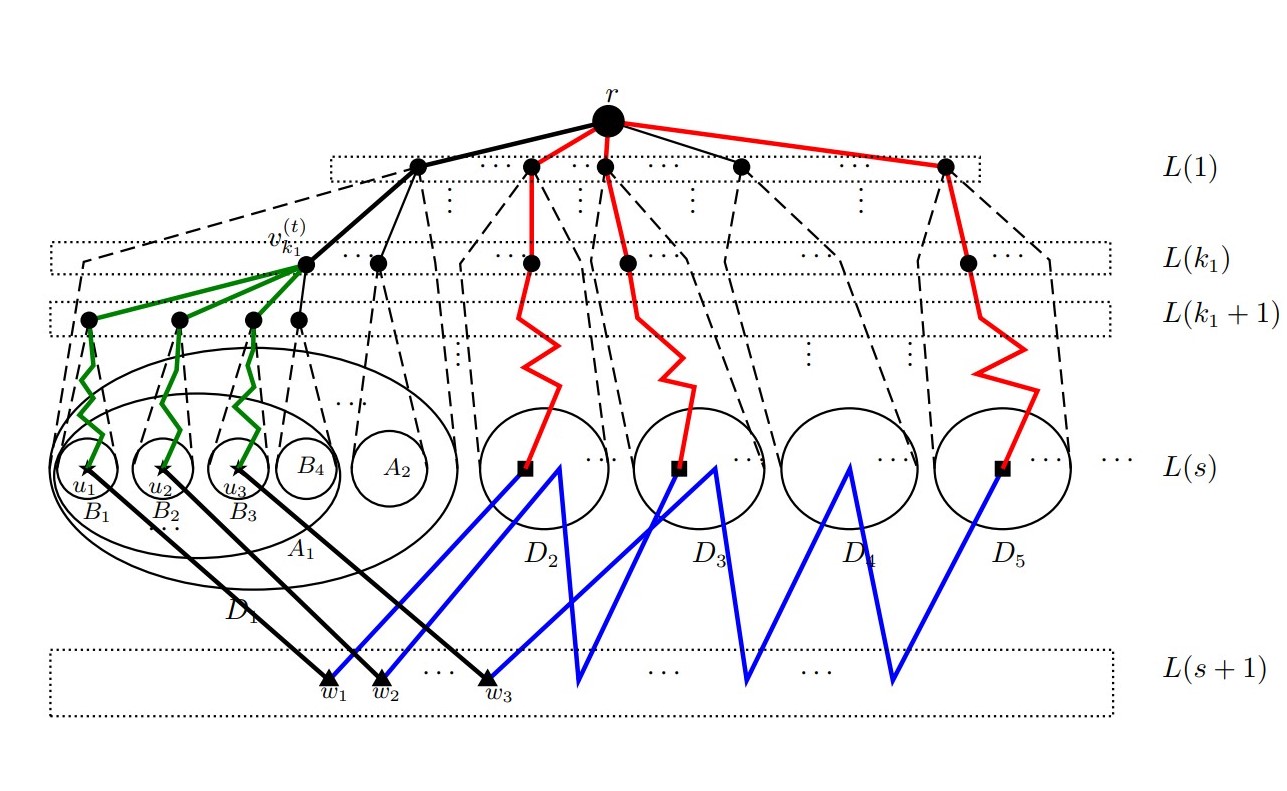}
	\caption{An illustration of the embedding of $\Theta_{k_1,\ldots,k_4}$. We plot three thick vertices $u_1$, $u_2$ and $u_3$ (stars) and three strong vertices $w_1$, $w_2$ and $w_3$ (triangles). We find 3 disjoint paths $P_1,P_2,P_3$ (blue paths) starting from $w_1$, $w_2$ and $w_3$, respectively. For the end vertices (squares), we find 3 internally disjoint paths $P_1',P_2', P_3'$ (red paths) starting from them and ending at $r$. There are also 3  internally disjoint paths $Q_1, Q_2, Q_3$ starting from $v_{k_1}^{(t)}$ and ending at $u_1$, $u_2$ and $u_3$ (green paths).      }\label{figure}
\end{figure}

\begin{proof}[The Second Part of Proof of Proposition~\ref{proposition_general}] 
By Lemma~\ref{first_type_induction}, there exists a subgraph $H$, which restricts to a regular almost-tree of type $(\frac{1}{M_{k_1+1}}n^{1/k^\ast}, k_1+2)$. Note that, in order to obtain $H$, we only have deleted vertices which were at distance at most $k_1+2$ with the root. Inductively,  we can suppose that, 
for $k_1+1 \leq s \leq k^\ast-1$, we have found 
a subgraph $H$ whose first $k^\ast+1$ layers are named as $\{L(j)\}_{j=0}^{k^\ast}$ which restricts to  
a regular almost-tree of type $(d,s+1)$, where $d=\frac{1}{M_s}n^{1/k^\ast}$. 
By Lemma~\ref{thick_thin_discussion}, $|L^{\text{thick}}(s)| \leq (\ell-2) \lfloor d\rfloor^{s-1}$. We delete $L^{ \text{thick}}(s)$ from $L(s)$, 
and repeat the argument in the appendix to extract a further subgraph, still called $H$.

 Now $H$ restricts to a regular tree of type $(\frac{1}{M_s'} n^{1/k^\ast},s)$, with $M_s'>M_s$,
in which the vertices in layer $L(s)$ are all thin and with $d$ children.
Then we recall the definition of $\mathfrak B^{(s+1)}$ as in Definition~\ref{badset_definition}. Here again, we stress that $\mathfrak B^{(s+1)}$ is a vertex subset of the original graph $G$, from the beginning of the proof of Proposition~\ref{proposition_general}.
Then, similar to the definitions of the sets $B_i$ made in (\ref{Bs+1definition}) and (\ref{else_Bi}), here we define 
\begin{align}
&	B_{s+1}=\mathfrak{B}^{(s+1)}\cap L(s+1).\\
&   B_s=\{v\in L(s)  \big |  \text{deg}_{H[L(s) \cup B_{s+1}]}(v)\geq  \frac{1}{2M_s}n^{1/k^\ast}\}. \\
&   B_i=\{v\in L(i)  \big |  \text{deg}_{H[L(i) \cup B_{i+1}]}(v)\geq  \frac{1}{2M'_s}n^{1/k^\ast}\}, \ \ i=s-1,\cdots,1.
\end{align}

Let us write  $C_2=|\Lambda(s)|$. If we consider the graph $R=H[B_{s+1}\cup B_s]$, then we can show that 
\begin{equation}
	|E(R)|<C_2|V(R)|
	\end{equation}
To see this, suppose otherwise, then 
$\Lambda(s)$ can be embedded in $R$ so that its center is sent to $L(s+1)$, and all its leaves are sent to pairwise distinct elements in $\mathfrak{D}$. This is a contradiction since we have deleted the thick vertices from the $s$th layer.

 Therefore $B_i$, $i=1,\cdots, s+1$, satisfy the following estimates.
 Firstly, 
\begin{equation}\label{B_second_induction_1}
	|B_s| \leq \frac{4M_s C_2}{n^{1/k^{\ast}}} |B_{s+1}|.
	\end{equation}
 Then the definitions give the following directly.
\begin{equation}\label{B_second_induction}
	|B_i| \leq \frac{2M'_s}{n^{1/k^{\ast}}}| B_{i+1}|
=O(n^{\frac{i-1}{k^\ast}}) \ll |L(i)|, \ \ i=s-1,\cdots, 1.
\end{equation}
In particular, we have 
\begin{equation}
	|B_1|=O(1).
\end{equation}

Now again, since $\Lambda(s)$ does not embed in $H[L(s)\cup L(s+1)]$ in a certain way, we have the following estimate. 
\begin{equation}\label{s+1layermanyvtx}
|L(s+1)| \geq \frac{n^{1/k^\ast}}{4C_2M_s}|L(s)| \geq \frac{1}{4C_2M_s'} \frac{1}{(M_s')^s}n^{(s+1)/k^\ast}.
\end{equation}

If $s\leq k^\ast-2$, we want to find a subgraph which restricts to a regular almost-tree of type $(\frac{1}{M_{s+1}}n^{1/k^\ast}, s+2)$. For this, we proceed in the following 3 steps, which are similar with the arguments in Lemma \ref{first_type_induction}. Firstly we delete $B_i$ from $L(i)$, for $i=1,\cdots, s+1$. By the definitions of $B_i$ and (\ref{B_second_induction_1}), 
	(\ref{B_second_induction}), every remaining vertex $v\in L(i)$, for $i=0,\cdots, s$,
	 has at least $\frac{1}{2M_s'}n^{1/k^\ast}$ children (actually, vertices in $L(s)$ have at least $\frac{1}{2M_s}n^{1/k^\ast}$ children). Every vertex $v\in L(s+1)$ has at least $\frac{1}{2C_0}n^{1/k^\ast}$ children. Then we delete some vertices from $L(i)$, $i=1,\cdots, s$, such that the resulting graph, still called $H$, restricts to a regular tree of type $(\frac{1}{2M_s'}n^{1/k^\ast}, s)$. 
\begin{enumerate}
	\item {\bf Grow a regular tree of type $(\frac{1}{M_{s+1}''} n^{1/k^\ast}, s+1)$ from a regular tree of type $(\frac{1}{2M_s'} n^{1/k^\ast}, s)$ for some larger constant $M_{s+1}''$.}
	
	By taking $C_0$ in Lemma $\ref{regular_tree}$ to be $2M_s$, condition (B) of Lemma \ref{regular_tree} is satisfied. 
Note that there is no strong vertex in $L(s+1)$, so $\Lambda(s)$ cannot embed in a certain way described eariler.
It means  condition (C) of Lemma \ref{regular_tree} is satisfied by taking $C_1$ there to be $C_2$. By Lemma \ref{regular_tree}, we can update $H$ such that it restricts to a regular tree of type $(\frac{1}{M_{s+1}''} n^{1/k^\ast}, s+1)$.  
	\item {\bf For each $v\in L(1)$, grow a regular tree of type $(\frac{1}{M_{s+1}}n^{1/k^\ast},s+1)$ from a regular tree of type $(\frac{1}{M_{s+1}''}n^{1/k^\ast},s)$,  regarding $v$ as the root.}

Just like in Step (2) in the proof of Lemma \ref{first_type_induction}, here we again inductively and alternatively construct regular trees and regular almost-trees from the bottom up. The only difference is that, at each step when the regular trees grow bigger, we need to delete more thick vertices to continue the process. 

For any vertex $v\in H$, let $H^{v}$ be the induced subgraph of $v$ and all its descendants. For $t=1,\cdots,k_1$, similar with Step (2) in the proof of Lemma \ref{first_type_induction}, repeatedly by Lemma \ref{first_embed} and Lemma \ref{regular_tree}, we can find constants $M_{s+1}^{(t)}$, such that for all vertices $v_i\in L_{H^v}^v(s-t)$, $i=1,\cdots,m$, $m=|V(L_{H^v}^v(s-t))|$,  $H^{v_i}$ restricts to a regular tree of type $(\frac{1}{M_{s+1}^{(t)}}n^{1/k^\ast}, t+1)$. For every $v'\in L_{H^v}^v(s-t-1)$, up to deleting some of $v_1,\cdots,v_m$ and their descendants, $H^{v'}$ restricts to a regular almost-tree of type $(\frac{1}{M_{s+1}^{(t)}}n^{1/k^\ast},t+2)$. In this way, we are able to grow regular trees and regular almost-trees alternatively, one level bigger at each step. Finally, when $t=k_1$, for every $v''\in L^v_{H^v}(s-k_1-1)$ and each of its children $v'$, we have that $H^{v'}$ restricts to a regular tree of type $(\frac{1}{M_{s+1}^{(k_1)}}n^{1/k^\ast},k_1+1)$ and therefore, $H^{v''}$ restricts to a regular almost-tree of type $(\frac{1}{M_{s+1}^{(k_1)}}n^{1/k^\ast},k_1+2)$.

For $t=k_1+1$, in $H^{v'}$ we define thick vertices in $L_{H^{v'}}^{v'}(k_1+1)\subset L_{H^v}^v(s)$ via Definition \ref{thick} by taking $s$ there equal to $k_1+1$. By Lemma~\ref{thick_thin_discussion}, the number of thick vertices for $H^{v'}$ is no more than $(\ell-1)(\frac{1}{M_{s+1}^{(k_1)}}n^{1/k^\ast})^{k_1}$. We then repeat the procedure at the beginning of {\it The Second Part of Proof of Proposition \ref{proposition_general}}. More precisely, in $H^{v'}$, we delete thick vertices from $L^{v'}_{H^{v'}}(k_1+1)$ and trim it into a smaller regular tree such that $H^{v'}$ restricts to a regular tree of type $(\frac{1}{M_{s+1}^{(k_1)'}}n^{1/k^\ast},k_1+1)$. We can also check that $H^{v'}$ satisfies condition (C) of Lemma \ref{regular_tree} by taking $C_1=|\Lambda(k_1+1)|$ for $\Lambda(k_1+1)$ cannot be embedded into $H^{v'}[L^{v'}_{H^{v'}}(k_1+1)\cup L^{v'}_{H^{v'}}(k_1+2)]$. By Lemma \ref{regular_tree}, $H^{v'}$ restricts to a regular tree of type $(\frac{1}{M_{s+1}^{(k_1+1)}}n^{1/k^\ast},k_1+2)$. We do the same procedure for every vertex $v'\in L^v_{H^v}(s-k_1-1)$,
	so that $H^{v'}$ restricts to a regular tree of type $(\frac{1}{M_{s+1}^{(k_1+1)}}n^{1/k^\ast},k_1+2)$. Note that for a vertex  $v''\in L^v_{H^v}(s-k_1-2)$, up to deleting some descendants until $L^v_{H^v}(s)$, $H^{v''}$ restricts to a regular almost-tree of type $(\frac{1}{M_{s+1}^{(k_1+1)}}n^{1/k^\ast},k_1+3)$. Also see Figure \ref{figure1}.

Suppose for any $t$, $k_1+1\leq t\leq s-1$, we found constants $M_{s+1}^{(t)}$, for any $v'\in L_{H^v}^v(s-t-1)$, $H^{v'}$ restricts to a regular almost-tree of type $(\frac{1}{M_{s+1}^{(t)}}n^{1/k^\ast},t+2)$. For $t+1$, in $H^{v'}$ we define thick vertices in $L_{H^{v'}}^{v'}(t+1)\subset L_{H^v}^v(s)$ and delete them by Lemma \ref{thick_thin_discussion}. We trim $H^{v'}$ into a smaller regular tree of type $(\frac{1}{M_{s+1}^{(t)'}}n^{1/k^\ast},t+1)$. Since there are no thick vertices, neither strong vertices, $H^{v'}[L_{H^{v'}}^{v'}(t+1)\cup L_{H^{v'}}^{v'}(t+2)]$ satisfies condition (C) in Lemma \ref{regular_tree} with respect to the choice of $C_1=|\Lambda(t+1)|$. By Lemma \ref{regular_tree}, we can find a constant $M_{s+1}^{(t+1)}$, such that $H^{v'}$ restricts to a regular tree of type $(\frac{1}{M_{s+1}^{(t+1)'}}n^{1/k^\ast},t+2)$. If $t<s-1$, we repeat this for every $v'\in L_{H^v}^v(s-t-1)$. Therefore, for any vertex $v''\in L_{H^v}^v(s-t-2)$, we can delete some of its children and descendants, so that 
$H^{v''}$ restricts to a regular almost-tree of type $(\frac{1}{M_{s+1}^{(t+1)'}}n^{1/k^\ast},t+3)$. The induction stops at $t=s-1$, $v'=v$ and  $H^v$ restricts to a regular tree of type $(\frac{1}{M_{s+1}}n^{1/k^\ast}, s+1)$, where $M_{s+1}=M_{s+1}^{(s)}$.

	\item {\bf Combine all the regular trees of type $(\frac{1}{M_{s+1}}n^{1/k^\ast},s+1)$ into a regular almost-tree of type $(\frac{1}{M_{s+1}}n^{1/k^\ast},s+2)$.} 
	
	From step (2), we actually obtain $\frac{1}{M_{s+1}''}n^{1/k^\ast}$ regular trees. We only keep $\frac{1}{M_{s+1}}n^{1/k^\ast}$ of them form a bigger regular almost-tree of type $(\frac{1}{M_{s+1}}n^{1/k^\ast}, s+2)$.  
\end{enumerate}
The above procedure finishes when $s=k^\ast-1$. 

In the case of $s=k^\ast-1$, we have 
\begin{equation}
	|L(k^\ast)| \geq  \frac{1}{4C_2M_{k^\ast-1}'} \frac{1}{(M_{k^\ast-1}')^{k^\ast-1}}n.
\end{equation}
We conclude by taking $M=4C_2(M_{k^\ast-1}')^{k^\ast}$.
\end{proof}

\section{Lower Bound for $\text{ex}(n , \Theta_{3,5,5})$}\label{proof_335}
With Theorem 1.2 at hand, the proof of Theorem~\ref{Thm335} reduces to the following proposition. 
Its proof is based on the construction given in the papers ~\cite{wenger1991extremal} and~\cite{Conlon_cycles_C4C6C10}. Here we include all the details of the proof for the convenience of the readers. 
\begin{prop} 
 $\text{ex}(n,\Theta_{3,5,5})=\Omega(n^{5/4})$.
\end{prop}
\begin{proof}
Let $\Bbb F_q$ be the finite field with $q$ elements, where $q$ is a prime power. Then we consider the 
$4$-dimensional vector space $\Bbb F_q^4$ over $\Bbb F_q$. 
For any $z\in \Bbb F_q$, we obtain a direction $v_z=(1,z,z^2,z^3)$, 
which can be thought of the "discretized moment curve". 
For any $x\in \Bbb F_q^4$, we define $l_{x,z}=\{x+yv_z | y\in \Bbb F_q\}$. Then define $L_z=\{l_{x,z} | x\in \Bbb F_q^4\}$ as the family of parallel lines with the same direction $v_z$. 
Define a bipartite graph $G=G(q)$ on the bipartition $P\cup L$, where $P=\Bbb F_q^4$
and $L= \bigcup_{z\in \Bbb F_q}  L_z$. Thus each part has $q^4$ elements. 
A pair $(p,\ell) \in P\times L$ belongs to $E(G)$ if and only if $p\in \ell$. Observing that each line contains $q$ elements, 
it follows that $G$ contains $n=2q^4$ vertices and $|E(G)|=q\times |L|=q^{5}=(\frac{n}{2})^{5/4}$. 
Next we show a lemma. 
\begin{lemma}\label{fact_2} 
Suppose $p_1\ell_1 p_2\ell_2 p_3\ell_3 p_4\ell_4p_1$ is a copy of $C_8$ in $G$.
Let $v_1,v_2,v_3,v_4$ denote the directions of the lines $\ell_1, \ell_2,\ell_3,\ell_4$, 
respectively. Then $v_1=v_3, v_2=v_4$, which are two distinct directions. 
\end{lemma}

\begin{proof}[Proof of Lemma~\ref{fact_2}]
Write $\Bbb Z/ 4\Bbb Z= \{0,1,2,3\}$.
Then for each $i\in \Bbb Z/4 \Bbb Z$, we have $p_{i+1}-p_i= a_i v_i$ for some 
$a_i\in \Bbb F_q \backslash \{0\}$. Then we have 
$\sum_{i=0}^3 a_i v_i= \sum_{i=0}^3 (p_{i+1}-p_i)=0.$
Write each $v_i=(1,z_i,z_i^2,z_i^3)$ for some $z_i\in \Bbb F_q$. 
The Vandermonde determinant then tells us that there must exist
$z_i=z_{i'}$ for two different indices $i$ and $i'$. 
Note that two consecutive lines $\ell_i$ and $\ell_{i+1}$
 cannot be parallel to each other since they intersect at one point. 
 Without loss of generality we find $v_1=v_3$ and clearly this vector does not belong to $\{v_2,v_4\}$.
Then we can combine these two terms together in the above equation system and repeat the argument. 
Finally we obtain $v_2=v_4$ and finish the proof. 
\end{proof}
Back to the proof of the proposition, the graph $\Theta_{3,5,5}$ consists of two vertices $w$ and $w'$, and pairwise disjoint three paths $P_0,P_1, P_2$ connecting them, such that, $P_0$ has length $3$, and each of $P_1$ and $P_2$ 
has length $5$. 
It suffices to show that $G$ is $\Theta_{3,5,5}$-free. 
Suppose for contradiction that one can embed $\Theta_{3,5,5}$ into $G$. 
Note that the two ends of the embedded path $P_0$ 
must be a point and a line respectively. So we can write the embedded image of $P_0$ as $p\ell p' \ell'$. Note that $\ell$ and $\ell'$ are not parallel because they share one point $p'$. 
For the paths $P_1$ and $P_2$, each of their embedded image starts from 
$p$ and ends at $\ell'$. 
The second vertex of $P_1$ embeds in a line called $\ell_1$ which is 
parallel to $\ell'$ by  Lemma~\ref{fact_2}.
Similarly, 
the second vertex of $P_2$ embeds in a line called $\ell_2$ which is also parallel to $\ell'$ by Lemma~\ref{fact_2}.
This is a contradiction since $\ell_1$ and $\ell_2$ are different lines and they contain the same point $p$. This contradiction shows that $G$ is $\Theta_{3,5,5}$-free. By varying $q$ and observing Bertrand's postulate that for any integer $n>1$, there exists at least one prime $p$ contained in the integer interval $(n,2n)$, the conclusion follows. 
\end{proof}

\section*{ACKNOWLEDGMENTS}
We thank Zixiang Xu and Jie Han for useful discussions. We also thank the anonymous referee for carefully reading our manuscript and providing many useful suggestions and even corrections. 
X-C. Liu is supported by Fapesp P\'os-Doutorado grant (Grant Number  2018/03762-2).  
\bibliographystyle{plain}
\addcontentsline{toc}{chapter}{Bibliography}
\bibliography{mixed_v5}

\begin{thebibliography}{10}

\bibitem{Bondy1974cycles}
John Bondy and Mikl{\'o}s Simonovits.
\newblock Cycles of even length in graphs.
\newblock {\em Journal of Combinatorial Theory, Series B}, 16(2):97--105, 1974.

\bibitem{Bukh2014random}
Boris Bukh.
\newblock Random algebraic construction of extremal graphs.
\newblock {\em Bulletin of the London Mathematical Society}, 47:839--945, 2015.

\bibitem{bukh2017bound}
Boris Bukh and Zilin Jiang.
\newblock A bound on the number of edges in graphs without an even cycle.
\newblock {\em Combinatorics, Probability \& Computing}, 26(1):1, 2017.

\bibitem{Bukh_theta}
Boris Bukh and Michael Tait.
\newblock Tur{\'a}n numbers of theta graphs.
\newblock {\em Combinatorics, Probability \& Computing}, 29(4):495--507, 2020.

\bibitem{Conlon2019graphs}
David Conlon.
\newblock Graphs with few paths of prescribed length between any two vertices.
\newblock {\em Bulletin of the London Mathematical Society}, 51(6):1015--1021,
  2019.

\bibitem{Conlon_cycles_C4C6C10}
David Conlon.
\newblock Extremal numbers of cycles revisited.
\newblock {\em The American Mathematical Monthly, to appear}, 2020.

\bibitem{erdos1965some}
Paul Erd{\H o}s and Mikl{\' o}s Simonovits.
\newblock On some extremal problems in graph theory.
\newblock {\em Colloquia Mathematica Societatisj{\'a}nos Bolyai},
  4.Combinatorial Theory and Its applications, 1969.

\bibitem{Faudree1983}
Ralph Faudree and Mikl{\'o}s Simonovits.
\newblock On a class of degenerate extremal graph problems.
\newblock {\em Combinatorica}, 3(1):83--93, 1983.

\bibitem{Furedi2013history}
Zolt{\'a}n F{\"u}redi and Mikl{\'o}s Simonovits.
\newblock The history of degenerate (bipartite) extremal graph problems.
\newblock In {\em Erd{\H{o}}s Centennial}, pages 169--264. Springer, 2013.

\bibitem{new_C2k_bound}
Zhiyang He.
\newblock a new upper bound on extremal number of even cycles.
\newblock {\em the Electronic Journal of Combinatorics}, 28, 2021.

\bibitem{Jiang2012turan}
Tao Jiang and Robert Seiver.
\newblock Tur{\'a}n numbers of subdivided graphs.
\newblock {\em SIAM Journal on Discrete Mathematics}, 26(3):1238--1255, 2012.

\bibitem{Turan1941}
P{\'a}l Tur{\'a}n.
\newblock On an extremal problem in graph theory.
\newblock {\em Mat.Fiz.Lapok (Hungarian)}, 3(48):436--452, 1941.

\bibitem{verstraete2019graphs}
Jacques Verstra{\"e}te and Jason Williford.
\newblock Graphs without theta subgraphs.
\newblock {\em Journal of Combinatorial Theory, Series B}, 134:76--87, 2019.

\bibitem{wenger1991extremal}
Rephael Wenger.
\newblock Extremal graphs with no c4's, c6's, or c10's.
\newblock {\em Journal of Combinatorial Theory, Series B}, 52(1):113--116,
  1991.

\end{thebibliography}

\section*{Appendix}\label{appendix}
\begin{proof}[Proof of Lemma~\ref{regular_tree}] 
We can assume $d$ is an integer, up to taking the floor function and re-defining it. 
Thus, for any $i=1,\cdots,s$, $|L(i)|= d^i$.
We construct  $G^\ast$ by extracting a sequence of subgraphs. Consider 
the induced subgraph $G[L(s) \cup L(s+1)]$ which has edge number at least $d |L(s)|$.
By conditions (B) and (C), we have
\begin{equation}
|L(s)|+|L(s+1) | \geq \frac{1}{C_1} 
 | E \big(G[L(s) \cup 
L(s+1)] \big )| \geq \frac{1}{C_1} d |L(s)|.
\end{equation}
So we have $|L(s+1)|\geq (\frac{d}{C_1}-1)|L(s)|\geq \frac{d}{2C_1}|L(s)|$.
We choose constant  $C_s'\geq 4C_1^2C_0^2$. 
With $m= |L(s)|$, $p= \frac{d}{2C_1}$, and $C=2C_0^2C_1$ as the input for
Lemma~\ref{extracting_stars}, it follows that 
$G[L(s) \cup L(s+1)]$ contains a subgraph, 
which is a disjoint union of more than
$\frac{1} {C_s'} d^s$ 
copies of $K_{d/C_s',1}$'s, each of which is centered at a vertex in $L(s)$. 
Then we keep only the stars and their ancestors and delete all other vertices and edges from all layers and denote by $G^{(s)}$ the resulting tree. Let $L_{G^{(s)}}(i)$, $i=1,\cdots,s+1$, denote the new layers.
In particular, there are 
 $\lfloor \frac{1}{C_s'}d^s \rfloor$ vertices in the layer $L_{G^{(s)}}(s)$.

Inductively, suppose for $t=s,s-1,\ldots, 2$, we have the graph $G^{(t)}$ 
and the constant $C_t'$, such that the induced subgraph 
$G^{(t)}
[\cup_{j=t}^{s+1}L_{G^{(t)}}(j)]$ consists of 
 $\lfloor \frac{1}{C'_t}d^t \rfloor$
disjoint regular trees of type $(\frac{1}{C'_t}d, s-t+1)$. In particular, 
the $t$-th layer $L_{G^{(t)}}(t)$ has size  $\lfloor \frac{1}{C_t'}d^t \rfloor$.
 Then we \textbf{claim} in $L_{G^{(t)}}(t-1)$, 
there are more than $\frac{1}{2C_t'}d^{t-1}$ vertices whose number of children is at least 
$\frac {1}{2C_t'}d$. 
 Otherwise, there are no more than 
 $\frac{1}{2C_t'}d^{t-1}$ vertices in $L_{G^{(t)}}(t-1)$ 
 which has at least $\frac {1}{2C_t'}d$ children.
Then the total number of vertices in $L_{G^{(t)}}(t)$ 
is less than 
\begin{equation}
		\frac{1}{2C_t'}d^{t-1} \times d+ 
		(1-\frac{1}{2C_t'}) d^{t-1} \times \frac{1}{2C_t'}d<\frac {1}{C_t'}d^t,
\end{equation}
 which is a contradiction.
 
 After the claim, we can define 
 $C_{t-1}'=2C_t'$ and then 
 find a set $S\subset L^{(t)}(t-1)$, 
 consisting of $\lfloor \frac{1}{C_{t-1}'}d^{t-1} \rfloor$ vertices. All the vertices in $S$ together with their  descendants have a subgraph which restricts to a disjoint union of $\lfloor \frac{1}{C_{t-1}'}d^{t-1} \rfloor$ regular trees of type $(\frac{1}{C_{t-1}'}d,s-t+2)$. The union of this subgraph and its ancestors forms $G^{(t-1)}$. The induction step is completed.  
 With the induction, we obtain sequences of trees $G^{(s)}\supset G^{(s-1)}\supset \cdots\supset G^{(1)}$ and sequence of constants $C_s'\leq C_{s-1}'\leq \cdots\leq C_1'$. Since all the vertices in $L_{G^{(1)}}(1)$ connect with the vertex $r$, $G^{(1)}$ is a regular tree of type $(\frac{1}{C_1'}d, s+1)$. 
We define $G^\ast=G^{(1)}$, whose restriction to the first $(s+1)$ layers is 
a regular tree of type $(\frac {1}{K}d,s+1)$, where $K=C_1'$, and the proof is completed.
\end{proof}

\end{document}